%% file: flip.tex
\documentclass[12pt]{amsart}
\usepackage[dvips]{graphicx}
\input{amssymb.sty}
\input xy
\xyoption{all}
\theoremstyle{plain}
  \newtheorem{thm}{Theorem}[section]
  \newtheorem{lem}[thm]{Lemma}
  
  \newtheorem{prop}[thm]{Proposition}
\theoremstyle{definition}

  \newtheorem{clm}[thm]{Claim}
\theoremstyle{remark}
  
\newcommand{\Ex}{\operatorname{Ex}}

\input{newcmds.tex}

\begin{document}
\title[Flips and variation of moduli of sheaves]
{Flips and variation of moduli scheme \\
of sheaves on a surface}
\author{Kimiko Yamada}
\email{kyamada@@math.kyoto-u.ac.jp}
\address{Department of mathematics, Kyoto University, Japan}
\thanks{Revised on Dec. 12, 2008.}
\maketitle
\begin{abstract}
Let $H$ be an ample line bundle on a non-singular projective surface $X$, and
$M(H)$ the coarse moduli scheme of rank-two $H$-semistable
sheaves with fixed Chern classes on $X$.
We show that if $H$ changes and passes through walls to
get closer to $K_X$, then $M(H)$ undergoes natural
flips with respect to canonical divisors.
When $X$ is minimal and $\kappa(X)\geq 1$,
this sequence of flips terminates in $M(H_X)$;
$H_X$ is an ample line bundle lying so closely to $K_X$
that the canonical divisor of $M(H_X)$ is nef.
Remark that so-called Thaddeus-type flips somewhat differ
from flips with respect to canonical divisors.
\end{abstract}
\section{Introduction}
Let $X$ be a non-singular projective surface over $\CC$, and
$H$ an ample line bundle on $X$.
Denote by $M(H)$ (resp. $\bar{M}(H)$) the coarse moduli scheme of 
rank-two $H$-stable (resp. $H$-semistable) sheaves
on $X$ with Chern class $(c_1,c_2)\in \Pic(X)\times \ZZ$.
We shall consider birational aspects of the problem how $\bar{M}(H)$ changes
as $H$ varies. \par
There is a union of hyperplanes $W\subset\Amp(X)$ called
$(c_1,c_2)$-walls in the ample cone $\Amp(X)$ such that
$\bar{M}(H)$ changes only when $H$ passes through walls.
Let $H_-$ and $H_+$ be ample line bundles separated by just one
wall $W$.
For $a\in(0,1)$ one can define $a$-semistability of sheaves
and the coarse moduli scheme $M(a)$ (resp. $\bar{M}(a)$) 
of rank-two $a$-stable (resp. $a$-semistable) sheaves
with Chern classes $(c_1,c_2)$
in such a way that $\bar{M}(\epsilon)=\bar{M}(H_+)$ and 
$\bar{M}(1-\epsilon)=\bar{M}(H_-)$
if $\epsilon>0$ is sufficiently small.
Let $a_-<a_+$ be parameters separated by only one miniwall $a_0$;
see Section \ref{scn:proof} about miniwalls.
Denote $\bar{M}_{\pm}=\bar{M}(a_{\pm})$ and $\bar{M}_0=\bar{M}(a_0)$.
There are natural morphisms $f_-: \bar{M}_-\rightarrow \bar{M}_0$ and
$f_+: \bar{M}_+ \rightarrow \bar{M}_0$.
After \cite{KM:birat},
let $f:X\rightarrow Y$ be a birational proper morphism such that
$K_X$ is $\QQ$-Cartier and $-K_X$ is $f$-ample, and that
the codimension of the exceptional set $\Ex(f)$ of $f$ is more than $2$.
We say a birational proper morphism $f_+: X_+\rightarrow Y$ is a 
{\it flip} of $f$
if (1) $K_{X_+}$ is $\QQ$-Cartier, (2) $K_{X_+}$ is $f_+$-ample and
(3) the codimension of the exceptional set $\Ex(f_+)$ is more than $2$.
We do not require that each contraction reduces
the Picard number by $1$.
The main result is the following.
\begin{thm}\label{thm:main}
Assume $c_2$ is so large that $\bar{M}_-$ and $\bar{M}_+$ are normal
and that the codimensions of
\begin{equation}\label{eq:defP}
 \bar{M}_{\pm}\supset P_{\pm}=\left\{ [E] \bigm| \text{$E$ is not 
$a_{\mp}$-semistable}\right\}
\end{equation}
and
\[\bar{M}_{\pm}\supset \Sing(\bar{M}_{\pm})=\{[E]\bigm| \dim\Ext^2(E,E)^0\neq 0 \}\]
are more than $2$.
Suppose $K_X$ does not lies in the wall $W$ separating $H_-$ and $H_+$,
and that $K_X$ and $H_-$ lie in the same connected components of 
$\NS(X)_{\RR}\setminus W$. (See the left figure below.)
Then the birational map
\begin{equation}\label{eq:flip}
\xymatrix{
 M_+  \ar[dr]_{f_-} \ar@{.>}[rr] & & M_- \ar[dl]^{f_+} \\
 & f_-(M_-)=f_+(M_+)\subset\bar{M}_0, & }
\end{equation}
where $f_-(M_-)$ is open in $\bar{M}_0$, is a flip.
\end{thm}
\includegraphics[height=3cm]{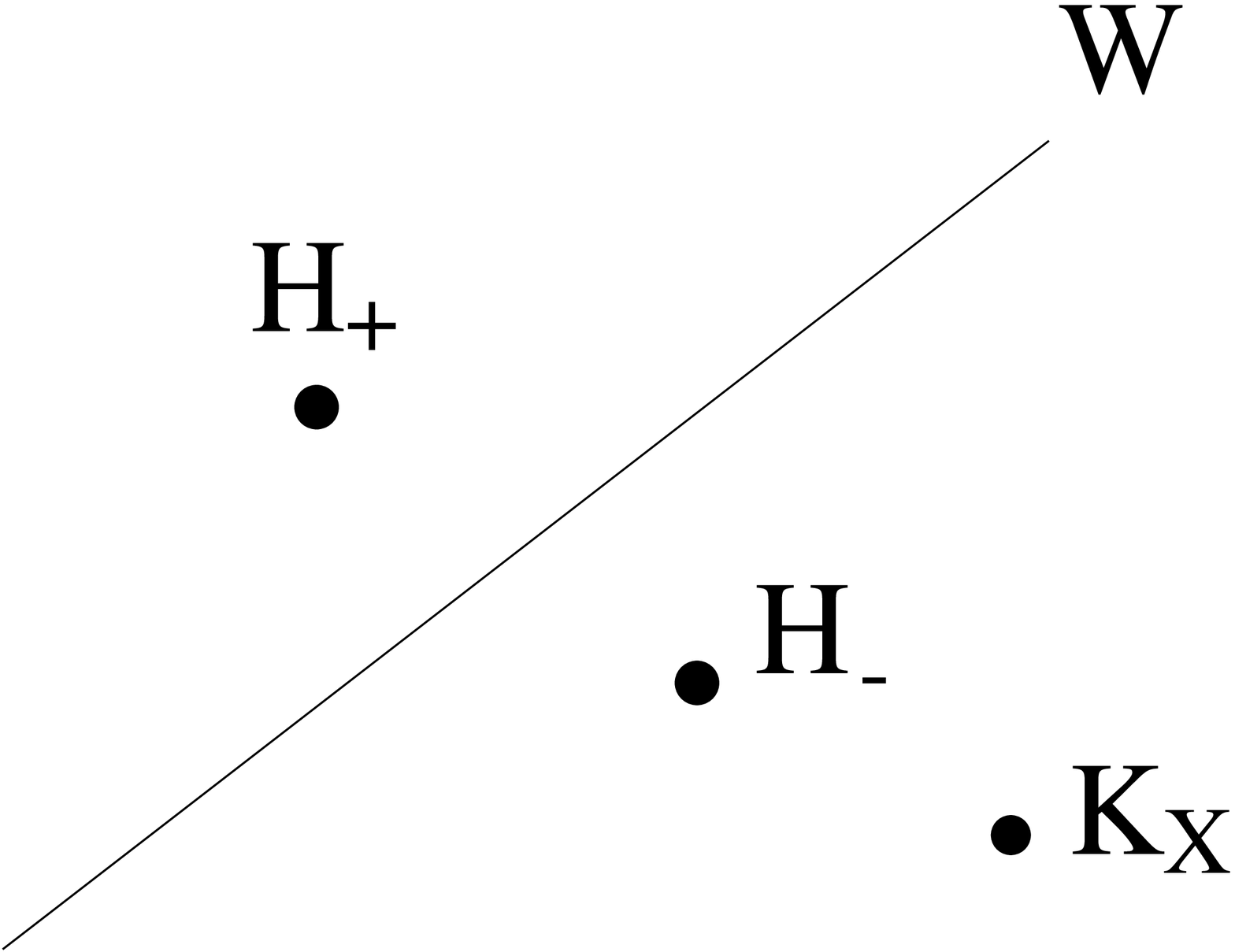}
\hspace{1cm}
\includegraphics[height=3cm]{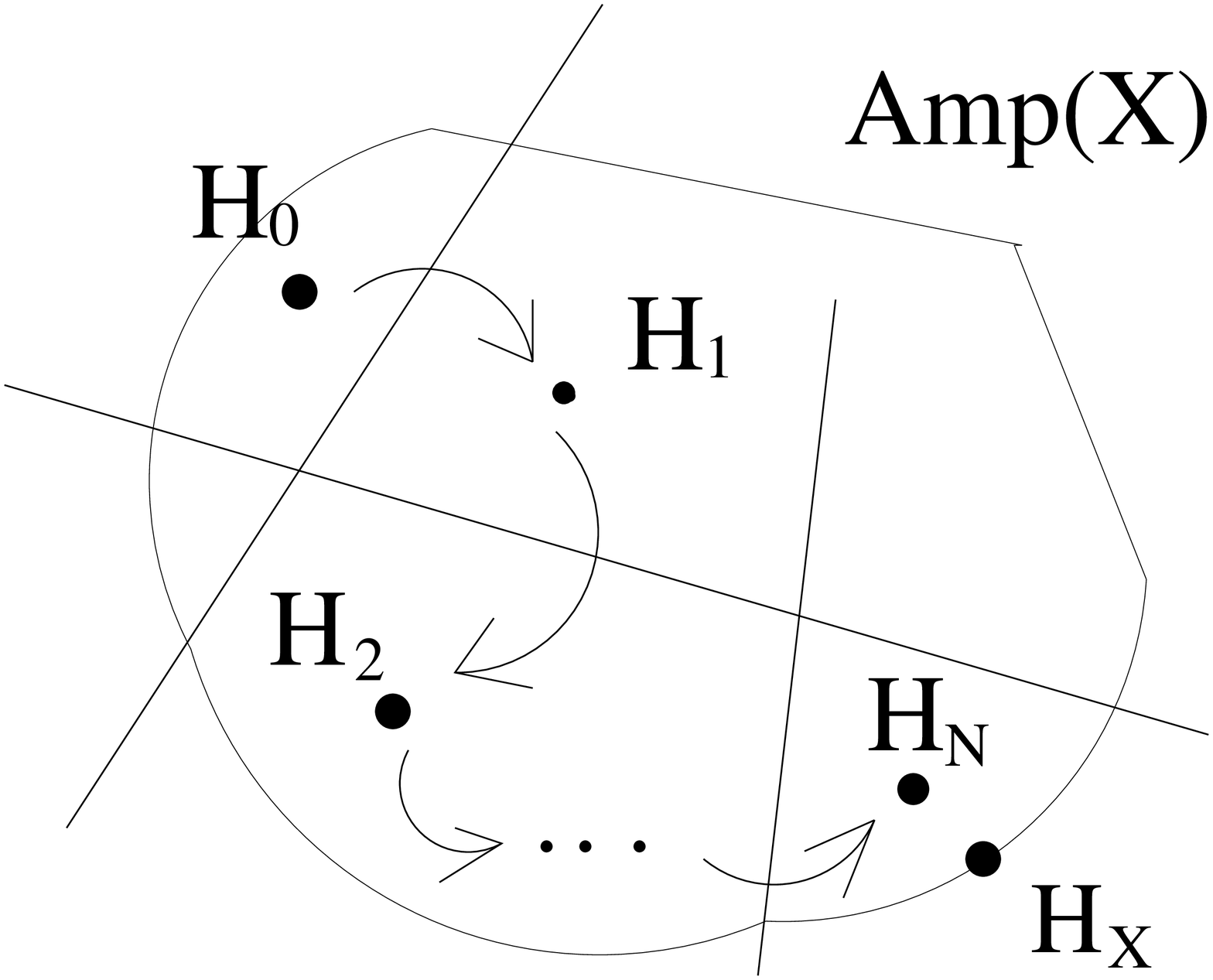} \\
Let us observe this theorem in case where $X$ is minimal and
$\kappa(X)\geq 1$.
When $X$ is of general type, $K_X$ is nef and big so there is an ample
line bundle, say $H_X$, such that no wall of type $(c_1,c_2)$ divides
$K_X$ and $H_X$ from the proof of \cite[Lem. 1.5]{MW:Mumford}.
Also when $X$ is elliptic one can find such $H_X$ from
\cite[p. 201]{Frd:holvb}.
Suppose $M(H)=\bar{M}(H)$ for any generic polarization $H$,
that is valid if $c_1=0$ and $c_2$ is odd for example.
Fix a compact polyhedral cone
${\mathcal S}$ in the closure of $\Amp(X)$
with ${\mathcal S}\cap \partial \Amp(X) \subset \RR K_X$.
If $c_2$ is sufficiently large, then
$M(H)$ are isomorphic in codimension one for $H\in {\mathcal S}$.
%
%
Hence when $H$ starts from a polarization $H_0$ 
and gets closer to $K_X$,
one gets flips
\[M(H)=M(a_0=\epsilon)\dots >M(a_1)\quad  \dots> M(a_{N'}=1-\epsilon)=M(H')\]
when $H$ passes through a wall to get closer to $K_X$ by Theorem
\ref{thm:main}, where $a_i$ are parameters lying in adjacent
miniwalls.
Hence the birational map $M(H) \dots > M(H')$ is decomposed into finite
sequence of natural flips.
After we repeat it finitely many times,
$H$ reaches $K_X$ and hence the sequence of birational morphisms
\[ M(H=H_0) \dots > M(H_1)  \quad \dots > M(H_N=H_X) \]
terminates in $M(K_X)$. (See the right figure above.)
It is known that the canonical divisor of $M(H_X)$ is nef.
Thus one can regard this ``natural'' process described in a
moduli-theoretic way as an analogy of minimal model program of $M(H)$, 
although it is unknown whether $M(K_X)$ admits only terminal
singularities. 
Note that $M(H)$ is of general type if $X$ is of general type,
$H^0(K_X)$ contains a reduced curve, and 
$\chi({\mathcal O}_X)+c_1^2(E)\equiv 0 (2)$ by \cite{Li:kodaira}.
\par
We mention some characteristics of this paper compared with
Thaddeus' work \cite{Tha:GITflip}, which carefully and widely
considered the variation of GIT quotients and linearizations.
First of all, the argument here goes independently.
As pointed out in \cite{MW:Mumford}, the rational map 
$M_-\rightarrow M_+$ is a Thaddeus-flip, that is, a rational map
which is an isomorphism in codimension $1$ and
comes from the variation of GIT quotient and linearization.
Thaddeus showed that
the rational map $X//G(-) \dots > X//G(+)$ is a flip with respect to
${\mathcal O}(1)\rightarrow X//G(+)$ at \cite[Theorem 3.3]{Tha:GITflip}.
Generally the relation about $K$-flip, that is,
flip with respect to the canonical divisor is not mentioned there.
So-called Thaddeus-flip is weaker than the flip defined above.
Moreover the birational map \eqref{eq:flip} is described concretely.
Indeed, birational maps $M_{\pm}\rightarrow M_0$ associate a sheaf
$E\in P_{\pm}$ with the graded sheaves $(F,G)$ of
the $H_0$-Jordan-H\"{o}lder filtration 
\[ 0 \longrightarrow F \longrightarrow E \longrightarrow G
     \longrightarrow 0; \]
see \cite[Prop. 3.14]{EG:variation}.
Moduli schemes $M_-$ and $M_+$ are connected by a natural blow-up and
a blow-down described in moduli theory; see \cite[Prop. 4.9]{Yam:Dthesis}.
\section{Proof of Theorem}\label{scn:proof}
There is a union of hyperplanes $W\subset\Amp(X)$ called
$(c_1,c_2)$-walls in the ample cone $\Amp(X)$ such that
$\bar{M}(H)$ changes only when $H$ passes through walls
(\cite{Qi:birational}).
Let $H_-$ and $H_+$ be ample line bundles separated by just one
wall $W$, and $H_0=\lambda H_- +(1-\lambda)H_+$ an ample line
bundle contained in $W$.
Let $H_-$ and $H_+$ be ample line bundles separated by just 
one wall $W$.
If $c_2$ is sufficiently large with respect to a compact subset
$S\subset\Amp(X)$ containing $H_{\pm}$, then $M_{\pm}$ are normal
and the codimension of $P_{\pm}\subset M_{\pm}$, which is defined at
\eqref{eq:defP}, are greater than $2$
from \cite{Li:kodaira} and \cite[Thm. 4.C.7]{HL:text}.
By \cite[Sect. 3]{EG:variation},
for a number $a\in [0,1]$ one can define the $a$-stability of a
torsion-free sheaf $E$ using
\[P_a(E(n))=\left. \{ (1-a)\chi(E(H_-)(nH_0))+ a\chi(E(H_+(nH_0)))
\}\right/ \rk(E).\]
There is the coarse moduli scheme $\bar{M}(a)$ of rank-two $a$-semistable
sheaves on $X$ with Chern classes $(c_1, c_2)$.
Denote by $M(a)$ its open subscheme of $a$-stable sheaves.
There is a finite numbers $a_i$ called miniwall
such that, as $a$ varies, $M(a)$ changes only when $a$ passes through
miniwalls.
Let $a_-<a_+$ be parameters separated by only one miniwall $a_0$,
and denote $\bar{M}_{\pm}=\bar{M}(a_{\pm})$ and
$\bar{M}_0=\bar{M}(a_0)$.
Since a rank-two $a_-$-semistable sheaf of type $(c_1,c_2)$ is
$a_0$-semistable,
there are natural morphisms $f_-: \bar{M}_-\rightarrow \bar{M}_0$ and
$f_+: \bar{M}_+ \rightarrow \bar{M}_0$
(\cite[Prop. 3.14]{EG:variation}).\par
Remark that the canonical divisors of $M_{\pm}$ are $\QQ$-Cartier.
Indeed, $M_-$ equal $R//\SL(N)$, where $R$ is a subscheme of
the Quot-scheme parameterizing quotient sheaves
${\mathcal O}_X(-M) \rightarrow E^-$ on $X$. Let $E^-_R$ be the universal
quotient sheaf on $X_R$.
From descent lemma \cite[Theorem 4.2.15]{HL:text},
$\det RHom_{X_R/R}(E^-_R,E^-_R)$ descends to a line bundle on $M_-$,
which we denote by $\det RHom_{X_{M_-}/M_-}(E^-, E^-)$.
It is known that $K_{M_-}|_{M_-\setminus \Sing(M_-)}$ equals
$\det RHom_{X_{M_-}/M_-}(E^-,E^-)$ from deformation theory.
Since $M_-$ is normal, we have 
\begin{equation}\label{eq:KM}
K_{M_-}=\det RHom_{X_{M_-}/M_-}(E^-, E^-), 
\end{equation}
so it is $\QQ$-Cartier.\par
Let $\eta$ be an element of
\[ A^+(W)=\left\{ \eta\in\NS(X) \bigm| \text{$\eta$ defines $W$,
$4c_2-c_1^2+\eta^2\geq 0$ and $\eta\cdot H_+>0$})\right\}. \]
After \cite[Definition 4.2]{EG:variation} we define
\[ T_\eta= M(1, (c_1+\eta)/2, n) \times M(1,(c_1-\eta)/2,m ),\]
where $n$ and $m$ are numbers defined by
\[ n+m=c_2-(c_1 -\eta^2)/4 \quad\text{and}\quad
   n-m= \eta\cdot(c_1-K_X)/2+(2a-1)\eta\cdot M,\]
and $M(1,(c_1+\eta)/2)$ is the moduli scheme of rank-one
torsion-free sheaves on $X$ with Chern classes $((c_1+\eta)/2,n)$.
If $F_{T_\eta}$ (resp. $G_{T_\eta}$) is the pull-back of a universal
sheaf of $M(1,(c_1+\eta)/2,n)$ (resp. $M(1,(c_1-\eta)/2,m)$) to
$X_{T_{\eta}}$, then we have the following.
\begin{prop}[\cite{Yam:Dthesis}, Section 5]\label{prop:PtoT}
We have isomorphisms
\begin{align}\label{eq:PtoT}
 P_- \simeq  & \underset{\eta\in A^+(W)}{\coprod} 
 \PP_{T_{\eta}}\left( Ext^1_{X_{T_{\eta}/T_{\eta}}}(F_{T_{\eta}},
G_{T_{\eta}}(K_X))\right)  \quad\text{and} \\
 P_+ \simeq  & \underset{\eta\in A^+(W)}{\coprod} 
 \PP_{T_{\eta}}\left( Ext^1_{X_{T_{\eta}/T_{\eta}}}(G_{T_{\eta}},
F_{T_{\eta}}(K_X))\right).
\end{align}
There are line bundles $L_i$ (resp. $L'_i$) on $P_-$ (resp. $P_+$)
with exact sequences
\begin{align*}
 0 \longrightarrow F_T\otimes L_1 \longrightarrow E^-_{M_-}|_{P_-} &
 \longrightarrow G_T\otimes L_2 \longrightarrow 0 \quad\text{and} \\
 0 \longrightarrow G_T\otimes L'_1 \longrightarrow E^+_{M_+}|_{P_+} &
 \longrightarrow F_T\otimes L'_2 \longrightarrow 0 
\end{align*}
such that $L_1\otimes L_2^{-1}={\mathcal O}_{P_-}(1)$, 
which means the tautological line bundle of the right side of \eqref{eq:PtoT},
and
$L'_1\otimes L_2^{'-1}={\mathcal O}_{P_+}(1)$. Here $E^-_{M_-}$ is a
universal family of $M_-$, which exists etale-locally.
\end{prop}
\begin{clm}
It holds that
\begin{align*}
K_{M_-}|_{P_-\underset{T}{\times} T_{\eta}}= &
 \,-(\eta\cdot K_X)\,{\mathcal O}_{P_-}(1)+ 
 \text{(some line bundle on $T$)},\quad \text{and} \\
K_{M_+}|_{P_+\underset{T}{\times} T_{\eta}}= & 
 \,(\eta\cdot K_X)\,{\mathcal O}_{P_+}(1)+ 
\text{(some line bundle on $T$)}.
\end{align*}
\end{clm}
\begin{proof}
Suppose that $A^+(W)=\{\eta\}$ for simplicity.
From the definition of walls, one can check 
$c_1(F_t)-c_1(G_t)=\eta$.
By Proposition \ref{prop:PtoT} and \eqref{eq:KM}, 
\begin{multline*}
K_{M_-}|_{P_-}= 
\det RHom_{X_{P_-}/P_-}(E^-_{M_-}|_{P_-}, E^-_{M_-}|_{P_-})= \\
\det RHom_{X_T/T}(F_T,F_T)  
  +\det (RHom_{X_T/T}(F_T,G_T)\otimes L_2\otimes L_1^{-1})\\
  +\det (RHom_{X_T/T}(G_T,F_T)\otimes L_1\otimes L_2^{-1}) 
  +\det RHom_{X_T/T}(G_T,G_T)\\
= -\chi(F_t,G_t)\cdot {\mathcal O}_{P_-}(1) +\chi(G_t,F_t)\cdot {\mathcal O}_{P_-}(1)
+(\text{line bundle on $T$}) \\
 = -(\eta\cdot K_X)\,{\mathcal O}_{P_-}(1)+ (\text{line bundle on $T$}).
\end{multline*}
One can calculate $K_{M_+}|_{P_+}$ similarly.
\end{proof}
Remark that, since $\eta\cdot H_+>0$, one can verify that 
$\eta\cdot K_X<0$ if and only if
$K_X$ does not lie in $W^{\eta}=W$, and $H_-$ and $K_X$ 
lie in the same connected components of 
$\NS(X)_{\RR}\setminus W$.
Therefore we obtain Theorem \ref{thm:main}.\par
We end with proving some facts in Introduction.
\begin{lem}
Fix a compact polyhedral cone
${\mathcal S}$ in the closure of $\Amp(X)$
with ${\mathcal S}\cap \partial \Amp(X) \subset \RR K_X$.
If $c_2$ is sufficiently large, then for $H\in {\mathcal S}$
$M(H)$ are isomorphic in codimension one.
\end{lem}
\begin{proof}
Remark that, for any $L\in \{ tK_X+(1-t)H_0 \bigm| 0\leq t < 1 \}$,
\[ |(K_X\cdot L)^2 /L^2 | \leq (K_X\cdot H_0)^2 / H_0^2 .\]
One can check this lemma in a similar way to 
\cite[Thm. 2.3]{Qi:birational}.
\end{proof}
\begin{clm}
 The canonical bundle of $M(H_X)$ is nef.
\end{clm}
\begin{proof}
From \cite[Prop. 8.3.1]{HL:text} $2K_{M(H_X)}=p_*(\Delta(E_{M(H_X)})\cdot K_X)$,
and $p_*(\Delta(E_{M(H_f)})\cdot H_X)$ is nef.
$M(H)=M(H_X)$ if a polarization $H$ is sufficiently close to $K_X$, so
the claim follows.
\end{proof}

%
\input{flip.bbl}

\end{document}

%% file: newcmds.tex
\newcommand{\Amp}{{\operatorname{Amp}}}

\newcommand{\CC}{{\mathbb{C}}}

\newcommand{\Ext}{\operatorname{Ext}}

\newcommand{\NS}{\operatorname{NS}}

\newcommand{\Pic}{\operatorname{Pic}}
\newcommand{\PP}{{\bf{P}}}

\newcommand{\QQ}{{\mathbb{Q}}}

\newcommand{\rk}{{\operatorname{rk}}}
\newcommand{\RR}{{\mathbb{R}}}

\newcommand{\Sing}{\operatorname{Sing}}
\newcommand{\SL}{\operatorname{SL}}

\newcommand{\ZZ}{{\mathbb{Z}}}

%% file: flip.bbl
\providecommand{\bysame}{\leavevmode\hbox to3em{\hrulefill}\thinspace}
\providecommand{\MR}{\relax\ifhmode\unskip\space\fi MR }
\providecommand{\MRhref}[2]{%
  \href{http://www.ams.org/mathscinet-getitem?mr=#1}{#2}
}
\providecommand{\href}[2]{#2}